\documentclass[11pt]{article}
\usepackage{amsfonts}
\usepackage{mathrsfs,dsfont,bbm}
\usepackage[latin1]{inputenc}
\usepackage{amsmath,amssymb}
\usepackage{latexsym}
\usepackage{epstopdf}
\usepackage{geometry}
\usepackage[active]{srcltx}
\usepackage{graphicx}
\usepackage[
bookmarks=true,         
bookmarksnumbered=true, 
colorlinks=true, pdfstartview=FitV, linkcolor=blue, citecolor=blue,
urlcolor=blue]{hyperref}

\newtheorem{theorem}{Theorem}[section]

\newtheorem{lemma}[theorem]{Lemma}

\numberwithin{equation}{section}
\newenvironment{proof}[1][Proof]{\noindent\textit{#1.} }{\hfill \rule{0.5em}{0.5em}}

\begin{document}
\date{\today}

\title{Derivative of self-intersection local time for multidimensional fractional Brownian motion}

\author
{Qian Yu$^1$\footnote{qyumath@163.com.}
\thanks{ Q. Yu is partially supported by  National Natural Science Foundation of China (12201294) and Natural Science Foundation of Jiangsu Province, China (BK20220865).} ~and Xianye Yu$^2$\footnote{xianyeyu@gmail.com} \thanks{ X. Yu is supported by National Natural Science Foundation of China (11701589).} \\
$^1$\small{School of Mathematics, Nanjing University of Aeronautics and Astronautics, Nanjing
211106, China }\\
$^2$\small{School of Statistics and Mathematics, Zhejiang gongshang University, Hangzhou, 310018, China}
}

\maketitle

\begin{abstract}
\noindent The existence condition $H<1/d$ for first-order derivative of self-intersection local time for $d\geq3$ dimensional fractional Brownian motion can be obtained in Yu \cite{Yu2021}. In this paper, we show a limit theorem under the non-existence critical condition $H=1/d$.
\vskip.2cm \noindent {\bf Keywords:} Self-intersection local time; Fractional Brownian motion; Limit theorem.

\vskip.2cm \noindent {\it Subject Classification: Primary 60G22;
Secondary 60J55.}
\end{abstract}

\section{Introduction}
Fractional Brownian motion (fBm) on $\mathbb{R}^d$ with Hurst parameter $H\in(0,1)$ is a $d$-dimensional centered Gaussian process $B^H=\{B_t^H, ~t\geq0\}$  with component processes being independent copies of a $1$-dimensional centered Gaussian process $B^{H,i}$, $i=1,2,\cdots,d$ and the covariance function given by
$$
\mathbb{E}[B_t^{H,i}B_s^{H,i}]=\frac{1}{2}\left[t^{2H}+s^{2H}-|t-s|^{2H}
\right].
$$
Note that $B_t^{\frac12}$ is a classical standard Brownian motion.
Let $D=\{(r,s): 0<r<s<t\}$. The self-intersection local time (SLT) of fBm was first investigated in Rosen \cite{Rosen 1987} and formally defined as
$$\alpha_t(y)=\int_{D}\delta(B^H_s-B^H_r-y)drds,$$
where $B^H$ is a fBm and $\delta$ is the Dirac delta function. It was further investigated in Hu \cite{Hu 2001}, Hu and Nualart \cite{Hu 2005}.
In particular, Hu and Nualart \cite{Hu 2005} showed its existence whenever $Hd<1$. Moreover,
$\alpha_t(y)$ is H\"{o}lder continuous in time of any order strictly less than $1-H$ which can be derived from Xiao \cite{Xiao 1997}.

The derivative of self-intersection local time (DSLT) for fBm was first considered in the works by Yan et al. \cite{Yan 2008} \cite{Yan 2009}, where the ideas  were based on Rosen \cite{Rosen2005}. The DSLT for fBm has two versions. One is extended by the Tanaka formula (see in Jung and  Markowsky \cite{Jung 2014}):
$$\widetilde{\alpha}'_t(y)=-H\int_{D}\delta'(B^H_s-B^H_r-y)(s-r)^{2H-1}drds.$$

The other is from the occupation-time formula (see Jung and  Markowsky \cite{Jung 2015}):
$$
\widehat{\alpha}'_t(y)=-\int_{D}\delta'(B^H_s-B^H_r-y)drds.
$$

Motivated by the $1$st order DSLT for fBm in Jung and Markowsky \cite{Jung 2015} and the $k$-th order derivative of intersection local time (ILT) for fBm  in Guo et al. \cite{Guo 2017}, we will consider the following $k$-th order DSLT  for fBm in this paper,

\begin{align*}
\widehat{\alpha}^{(k)}_t(y)&=\frac{\partial^k}{\partial y_1^{k_1}\cdots \partial y_d^{k_d}}\int_{D}\delta(B^H_s-B^H_r-y)drds\\
&=(-1)^{|k|}\int_{D}\delta^{(k)}(B^H_s-B^H_r-y)drds,
\end{align*}
where $k=(k_1,\cdots,k_d)$ is a multi-index with all $k_i$ being nonnegative integers and $|k|=k_1+k_2+\cdots+k_d$, $\delta$ is the Dirac delta function of $d$ variables and $\delta^{(k)}(y)=\frac{\partial^k}{\partial y_1^{k_1}\cdots \partial y_d^{k_d}}\delta(y)$ is the $k$-th order partial derivative of $\delta$.

Set
$$f_\varepsilon(x)=\frac1{(2\pi\varepsilon)^{\frac d2}}e^{-\frac{|x|^2}{2\varepsilon}}=\frac1{(2\pi)^d}\int_{\mathbb{R}^d}e^{\iota\langle p,x\rangle}e^{-\varepsilon \frac{|p|^2}{2}}dp,$$
where $\langle p,x\rangle=\sum_{j=1}^dp_jx_j$ and $|p|^2=\sum_{j=1}^dp_j^2$.

Since the Dirac delta function $\delta$ can be approximated by $f_\varepsilon(x)$, we approximate $\delta^{(k)}$ and $\widehat{\alpha}_t^{(k)}(y)$ by
$$f^{(k)}_\varepsilon(x)=\frac{i^{|k|}}{(2\pi)^d}\int_{\mathbb{R}^d}p_1^{k_1}\cdots p_d^{k_d}e^{\iota\langle p,x\rangle}e^{-\varepsilon \frac{|p|^2}{2}}dp$$
and
\begin{equation}\label{sec1-eq1.2}
\widehat{\alpha}^{(k)}_{t,\varepsilon}(y)=(-1)^{|k|}\int_{D}f^{(k)}_\varepsilon(B^H_s-B^H_r-y)drds,
\end{equation}
respectively.

If $\widehat{\alpha}^{(k)}_{t,\varepsilon}(y)$ converges to a random variable in $L^p$ as $\varepsilon\to0$, we denote the limit by $\widehat{\alpha}_t^{(k)}(y)$ and call it the $k$-th DSLT of $B^H$.

Recently,  Yu \cite{Yu2021} given a existence condition of $\widehat{\alpha}_t^{(k)}(y)$.
\begin{theorem} \label{sec1-th L2}\cite{Yu2021}
For $0<H<1$ and $\widehat{\alpha}^{(k)}_{t,\varepsilon}(y)$ defined in \eqref{sec1-eq1.2}, let $\#:=\#\{k_i ~is ~odd, ~i=1, 2, \cdots d\}$ denotes the odd number of $k_i$, for $i=1, 2, \cdots, d$. If $H<\min\{\frac2{2|k|+d},\frac{1}{|k|+d-\#}, \frac1d\}$ for $|k|=\sum_{j=1}^dk_j$, then  $\widehat{\alpha}^{(k)}_{t}(0)$ exists in $L^2$.
\end{theorem}

Note that, if $|k|=1$, the existence condition is $H<1/d$, and $Hd=1$ is the critical condition of $\widehat{\alpha}_t^{(k)}(y)$. When $d=2$, the critical case is $H=1/2$, Markowsky \cite{Mar2008} shown the limit theorem.

\begin{theorem} \label{sec1-th.d=2}\cite{Mar2008}
For $\widehat{\alpha}^{(k)}_{t,\varepsilon}(y)$  defined in \eqref{sec1-eq1.2} with $y=0$. Suppose that $H=\frac{1}{2}$, $d=2$ and $|k|=1$,
then as $\varepsilon\to0$,
$$
\Big(\log1/\varepsilon\Big)^{-1}\widehat{\alpha}^{'}_{t,\varepsilon}(0)\overset{law}{\to} N\left(0,\frac{5t}{64\pi^2\sqrt{2}}\right).
$$
\end{theorem}

In this paper, we will consider the case of $\frac1H=d\geq3$ and $|k|=1$, and study a limit theorem for  $\widehat{\alpha}^{'}_{t,\varepsilon}(0)$.

\begin{theorem} \label{sec1-th.d=3-2}
For $\widehat{\alpha}^{(k)}_{t,\varepsilon}(y)$  defined in \eqref{sec1-eq1.2} with $y=0$. Suppose that $H=\frac{1}{d}$, $d\geq3$ and $|k|=1$,
then as $\varepsilon\to0$, we have
$$
\Big(\varepsilon^{-\frac1H}\log1/\varepsilon\Big)^{H-\frac12}\widehat{\alpha}^{'}_{t,\varepsilon}(0)\overset{law}{\to}N(0,\sigma^2),
$$
where
$\sigma^2=\frac{2Hd^2t^{3-4H}}{(2\pi)^d(1-2H)^2}$.
\end{theorem}

When $|k|=1$, under the condition $H>1/d$, the behavior of $\widehat{\alpha}^{'}_{t,\varepsilon}(0)$ as $\varepsilon\to0$ is also of interest. One
would expect a central limit theorem to exist, but this remains unproved. Nevertheless, we venture the following conjecture

(1) If $H=\frac2{d+2}>\frac1d$ and $d\geq3$, $(\log\frac1\varepsilon)^{\gamma_1(H)}\widehat{\alpha}^{'}_{t,\varepsilon}(0)$ converges in distribution to a normal law for some $\gamma_1(H)<0$;

(2) If $H>\frac12\geq\frac2{d+2}$ and $d\geq2$, $\varepsilon^{\gamma_2(H)}\widehat{\alpha}^{'}_{t,\varepsilon}(0)$ converges in distribution to a normal law for some $\gamma_2(H)>0$;

(3) If $\frac2{d+2}<H<\frac12$ and $d\geq3$, $\varepsilon^{\gamma_3(H)}(\log\frac1\varepsilon)^{\gamma_4(H)}\widehat{\alpha}^{'}_{t,\varepsilon}(0)$ converges in distribution to a normal law for some $\gamma_3(H)>0$ and $\gamma_4(H)<0$.

The paper has the following structure. We present some preliminary lemmas in Section 2. Section 3 is to prove the main result.
Throughout this paper, if not mentioned otherwise, the letter $C$, with or without a subscript,
denotes a generic positive finite constant and may change from line to line.

\section{Preliminaries}

Next, we present two basic lemmas, which will be used in Section 3.

The next lemma gives the bounds on the quantity of $\lambda\rho-\mu^2$, which could be obtained from the Appendix B in \cite{Jung 2014} or the Lemma 3.1 in \cite{Hu 2001}.

\begin{lemma} \label{sec3-lem3.2}
Let $$\lambda=|s-r|^{2H}, ~~\rho=|s'-r'|^{2H},$$
and
$$\mu=\frac12\Big(|s'-r|^{2H}+|s-r'|^{2H}-|s'-s|^{2H}-|r-r'|^{2H}\Big).$$

\textbf{Case (i)} Suppose that $D_1=\{(r,r',s,s')\in[0,t]^4 ~|~ r<r'<s<s'\}$, let $r'-r=a$, $s-r'=b$, $s'-s=c$. Then, there exists a constant $K_1$ such that
$$\lambda\rho-\mu^2\geq K_1\,\left((a+b)^{2H}c^{2H}+a^{2H}(b+c)^{2H}\right)$$
and
$$2\mu=(a+b+c)^{2H}+b^{2H}-a^{2H}-c^{2H}.$$

\textbf{Case (ii)} Suppose that $D_2=\{(r,r',s,s')\in[0,t]^4 ~|~ r<r'<s'<s\}$, let $r'-r=a$, $s'-r'=b$, $s-s'=c$. Then, there exists a constant $K_2$ such that
$$\lambda\rho-\mu^2\geq K_2\,b^{2H}\left(a^{2H}+c^{2H}\right)$$
and
$$2\mu=(a+b)^{2H}+(b+c)^{2H}-a^{2H}-c^{2H}.$$

\textbf{Case (iii)} Suppose that $D_3=\{(r,r',s,s')\in[0,t]^4 ~|~ r<s<r'<s'\}$, let $s-r=a$, $r'-s=b$, $s'-r'=c$. Then, there exists a constant $K_3$ such that
$$\lambda\rho-\mu^2\geq K_3(ac)^{2H}$$
and
$$2\mu=(a+b+c)^{2H}+b^{2H}-(a+b)^{2H}-(c+b)^{2H}.$$
\end{lemma}

\bigskip

\begin{lemma} \label{sec3-lem-chaosdec}
Let $\widehat{\alpha}^{(k)}_{t,\varepsilon}(y)$ be defined in \eqref{sec1-eq1.2}, then we have the Wiener chaos expansion for $|k|=1$,
\begin{align*}
\widehat{\alpha}^{(k)}_{t,\varepsilon}(0)=\sum_{q=1}^{+\infty}I_{2q-1}(f_{2q-1,\varepsilon}).
\end{align*}

(i)If $d=2$, $f_{2q-1,\varepsilon}\in (\mathcal{H}^2)^{\otimes(2q-1)}$
\begin{align*}
f_{2q-1,\varepsilon}=\beta_q\int_{0<r<s<t}(|s-r|^{2H}+\varepsilon)^{-q-1}\prod_{j=1}^{2q-1}\mathds{1}_{[r,s]}(x_j)drds,
\end{align*}
where
$\beta_q=\frac{(-1)^q}{(2q-1)!\pi}\frac{(2q_1)!(2q_2)!}{(q_1)!(q_2)!2^q}$ and $q=q_1+q_2\geq1$.

(ii)If $d\geq3$, $f_{2q-1,\varepsilon}\in (\mathcal{H}^d)^{\otimes(2q-1)}$
\begin{align*}
f_{2q-1,\varepsilon}=\beta_{q,d}\int_{0<r<s<t}(|s-r|^{2H}+\varepsilon)^{-q-d/2}\prod_{j=1}^{2q-1}\mathds{1}_{[r,s]}(x_j)drds,
\end{align*}
where
$\beta_{q,d}=\frac{(-1)^qd}{(2q-1)!(2\pi)^{d/2}}\frac{(2q_1)!\cdots(2q_d)!}{(q_1)!\cdots(q_d)!2^q}$ and $q=q_1+\cdots q_d\geq1$.
\end{lemma}

\begin{proof}
\textbf{For $d=2$, $|k|=1$,}
\begin{align*}
\widehat{\alpha}^{(k)}_{t,\varepsilon}(0)&=\frac{\iota}{(2\pi)^2}\int_0^t\int_0^s\int_{\mathbb{R}^2}e^{\iota\langle \xi,B_s^H-B_r^H\rangle}(\xi_1+\xi_2)e^{-\varepsilon|\xi|^2/2}d\xi drds\\
&=\sum_{q=1}^{+\infty}I_{2q-1}(f_{2q-1,\varepsilon}),
\end{align*}
where $f_{2q-1,\varepsilon}\in (\mathcal{H}^2)^{\otimes(2q-1)}$ and
$$f_{n,\varepsilon}=\frac1{n!}\int_{0<r<s<t}\mathbb{E}[D_{x_1,\cdots,x_n}^{i_1,\cdots,i_n}f_{\varepsilon}(B_s^H-B_r^H)]drds$$
with $i_j\in{1,2}$, $x_j\in[0,t]$ for all $j=1,2,\cdots,n$.

Let us compute the expectation
$$\mathbb{E}[D_{x_1,\cdots,x_n}^{i_1,\cdots,i_n}f_{\varepsilon}(B_s^H-B_r^H)]=\mathbb{E}[\partial^{i_1}\cdots\partial^{i_n}f_{\varepsilon}(B_s^H-B_r^H)]\prod_{j=1}^n\mathds{1}_{[r,s]}(x_j)$$
and
\begin{align*}
&\mathbb{E}[\partial^{i_1}\cdots\partial^{i_n}f_{\varepsilon}(B_s^H-B_r^H)]\\
&=\frac{\iota^{n+1}}{(2\pi)^2}\int_{\mathbb{R}^2}(\xi_1+\xi_2)(\xi_{i_1}\xi_{i_2}\cdots\xi_{i_n})\mathbb{E}[e^{\iota\langle \xi,B_s^H-B_r^H\rangle}]e^{-\varepsilon|\xi|^2/2}d\xi\\
&=\frac{\iota^{n+1}}{(2\pi)^2}\int_{\mathbb{R}^2}(\xi_1+\xi_2)(\xi_{i_1}\xi_{i_2}\cdots\xi_{i_n})e^{-\frac12(|s-r|^{2H}+\varepsilon)|\xi|^2}d\xi\\
&=2(\iota)^{n+1}(2\pi)^{-1}(|s-r|^{2H}+\varepsilon)^{-1-\frac{n+1}{2}}\mathbb{E}[X_1X_{i_1}X_{i_2}\cdots X_{i_n}],
\end{align*}
where
\begin{equation*}
\mathbb{E}[X_1X_{i_1}X_{i_2}\cdots X_{i_n}]=
\begin{cases}
    \frac{(2m_1)!(2m_2)!}{(m_1)!(m_2)!2^m}, &\mbox{\text{if} $n=2(m_1+m_2)-1$, the number of $i_k=1$ is $2m_1-1$}\\
    & \mbox{ and the number of $i_k=2$ is $2m_2$,}\\
    0, &\mbox{\text{otherwise}.}
   \end{cases}
\end{equation*}

For $n=2q-1=2(q_1+q_2)-1$ with $m_1=q_1$ and $m_2=q_2$, we have
\begin{align*}
f_{2q-1,\varepsilon}=\beta_q\int_{0<r<s<t}(|s-r|^{2H}+\varepsilon)^{-q-1}\prod_{j=1}^{2q-1}\mathds{1}_{[r,s]}(x_j)drds,
\end{align*}
where
$\beta_q=\frac{(-1)^q}{(2q-1)!\pi}\frac{(2q_1)!(2q_2)!}{(q_1)!(q_2)!2^q}$.

\textbf{For $d\geq3$ and $|k|=1$,}
\begin{align*}
\widehat{\alpha}^{(k)}_{t,\varepsilon}(0)=\sum_{q=1}^{+\infty}I_{2q-1}(f_{2q-1,\varepsilon}),
\end{align*}
where
$$f_{n,\varepsilon}=\frac{(\iota)^{n+1}}{n!}\frac{d}{(2\pi)^{d/2}}\int_{0<r<s<t}(|s-r|^{2H}+\varepsilon)^{-(n+1/2-d/2)}\prod_{j=1}^{n}\mathds{1}_{[r,s]}(x_j)drds\cdot\mathbb{E}[X_1X_{i_1}X_{i_2}\cdots X_{i_n}],$$
and
\begin{equation*}
\mathbb{E}[X_1X_{i_1}X_{i_2}\cdots X_{i_n}]=
\begin{cases}
    \frac{(2m_1)!\cdots(2m_d)}{(m_1)!\cdots(m_d)!2^m}, &\mbox{\text{if} $n=2(m_1+\cdots+m_2)-1$, the number of $i_k=1$ is $2m_1-1$}\\
    & \mbox{ and the number of $i_k=\ell$ is $2m_\ell$ for $\ell=2,\cdots,d$,}\\
    0, &\mbox{\text{otherwise}.}
   \end{cases}
\end{equation*}

Thus,
\begin{align*}
f_{2q-1,\varepsilon}=\beta_{q,d}\int_{0<r<s<t}(|s-r|^{2H}+\varepsilon)^{-q-d/2}\prod_{j=1}^{2q-1}\mathds{1}_{[r,s]}(x_j)drds,
\end{align*}
where
$\beta_{q,d}=\frac{(-1)^qd}{(2q-1)!(2\pi)^{d/2}}\frac{(2q_1)!\cdots(2q_d)!}{(q_1)!\cdots(q_d)!2^q}$ and $q=q_1+\cdots q_d\geq1$.
\end{proof}

\begin{lemma} \label{sec3-lem-intlog}
If $Hd=1$, as $\varepsilon\to0$, we have

(i) $$\int_0^{\varepsilon^{-\frac1H}}x^{H-\frac12}(1+x^{H})^{-\frac{d}2-1}dx=O\left(\log\frac1{\varepsilon}\right)$$
and

(ii)$$\int_0^1x^{2H}(\varepsilon+x^{2H})^{-\frac{d}2-1}dx=O\left(\log\frac1{\varepsilon}\right).$$
\end{lemma}

\begin{proof}
For (i), by L'H\^{o}spital's rule, we have
\begin{align*}
\lim_{\varepsilon\to0}\frac1{\log\frac1{\varepsilon}}\int_0^{\varepsilon^{-\frac1H}}x^{H-\frac12}(1+x^{H})^{-\frac{d}2-1}dx
&=\lim_{\varepsilon\to0}\frac1H\varepsilon^{-1-\frac1{2H}}(1+\varepsilon^{-1})^{-\frac{d}2-1}\\
&=\lim_{\varepsilon\to0}\frac1H(\varepsilon+1)^{-\frac{d}2-1}\\
&=\frac1H,
\end{align*}
where we use the condition $Hd=1$ in the second equality.

For (ii), take the variable transformation $x=x\varepsilon^{-\frac1{2H}}$,
\begin{align*}
\lim_{\varepsilon\to0}\frac1{\log\frac1{\varepsilon}}\int_0^1x^{2H}(\varepsilon+x^{2H})^{-\frac{d}2-1}dx
&=\lim_{\varepsilon\to0}\frac1{\log\frac1{\varepsilon}}\int_0^{\varepsilon^{-\frac1{2H}}}x^{2H}(1+x^{2H})^{-\frac{d}2-1}dx\\
&=\lim_{\varepsilon\to0}\frac1{2H}(\varepsilon+1)^{-\frac{d}2-1}\\
&=\frac1{2H},
\end{align*}
where we use L'H\^{o}spital's rule and the condition $Hd=1$ in the second equality.

\end{proof}

\section{Proof of Theorem \ref{sec1-th.d=3-2}}
In this section, the proof of  Theorem \ref{sec1-th.d=3-2}  is taken into account, we will consider the case of $H>1/d$, $d\geq3$ and $|k|=1$.
$\widehat{\alpha}^{'}_{t,\varepsilon}(0)$ has the following chaos decomposition
$$
\widehat{\alpha}^{'}_{t,\varepsilon}(0)=\sum_{q=1}^\infty I_{2q-1}(f_{2q-1,\varepsilon}),
$$
where
$$f_{2q-1,\varepsilon}(x_1,\cdots,x_{2q-1})=\int_{D}f_{2q-1,\varepsilon,s,r}(x_1,\cdots,x_{2q-1})drds$$
with $D=\{(r,s): 0<r<s<t\}$.

For $q=1$,
\begin{equation}\label{sec2-eq2.3}
\mathbb{E}\Big[\Big|I_1(f_{1,\varepsilon})\Big|^2\Big]=\int_{D^2}\langle f_{1,\varepsilon,s_1,r_1},f_{1,\varepsilon,s_2,r_2}\rangle_{\mathfrak{H}}dr_1dr_2ds_1ds_2,
\end{equation}
where $\mathcal{H}^d$ is the Hilbert space obtained by taking the completion of the step functions endowed with the inner product
$$\langle \mathds{1}_{[a,b]}, \mathds{1}_{[c,d]}\rangle_{\mathcal{H}^d}:=\mathbb{E}[(B_b^H-B_a^H)(B_d^H-B_c^H)].$$

For $q>1$, we have to describe the terms $\langle f_{2q-1,\varepsilon,s_1,r_1},f_{2q-1,\varepsilon,s_2,r_2}\rangle_{(\mathfrak{H}^d)^{\otimes(2q-1)}}$,
where $(\mathcal{H}^{d})^{\otimes(2q-1)}$ is the $(2q-1)$-th tensor product of $\mathcal{H}^d$. For every $x, ~u_1, ~u_2>0$, we define
$$\mu(x,u_1,u_2)=|\mathbb{E}[B_{u_1}^H(B_{x+u_2}^H-B_x^H)]|.$$

Then
\begin{equation}\label{sec2-eq2.ine}
\langle f_{2q-1,\varepsilon,s_1,r_1},f_{2q-1,\varepsilon,s_2,r_2}\rangle_{(\mathcal{H}^d)^{\otimes(2q-1)}}=\beta_{q,d}^2G^{(q,d)}_{\varepsilon,r_2-r_1}(s_1-r_1,s_2-r_2),
\end{equation}
where
$$G^{(q,d)}_{\varepsilon,x}(u_1,u_2)=\Big(\varepsilon+u_1^{2H}\Big)^{-\frac{d}2-q}\Big(\varepsilon+u_2^{2H}\Big)^{-\frac{d}2-q}\mu(x,u_1,u_2)^{2q-1}.$$

Before we give the proof of the main result, we give some useful lemmas below. In the sequel, we just consider the case $H=\frac1d, d\geq3$.

\begin{lemma} \label{sec3-lem3.4}
For $\widehat{\alpha}^{'}_{t,\varepsilon}(0)$ defined in \eqref{sec1-eq1.2}, then
$$
\lim_{\varepsilon\to0}\mathbb{E}\Big[\Big|\Big(\varepsilon^{-\frac1H}\log1/\varepsilon\Big)^{H-\frac12}\widehat{\alpha}^{'}_{t,\varepsilon}(0)\Big|^2\Big]=\sigma^2.
$$
\end{lemma}

\begin{proof}
From Lemma 5.1 in Jaramillo and Nualart \cite{Jaramillo 2017}, we can see
$$\mathbb{E}[XYf_{\varepsilon}(X)f_{\varepsilon}(Y)]=(2\pi)^{-1}\varepsilon^2|\varepsilon I+A|^{-\frac32}A_{1,2},$$
where $(X,Y)\in\mathbb{R}\times\mathbb{R}$ is a jointly Gaussian vector with mean zero and covariance $A=(A_{i,j})_{i,j=1,2}$. Then for any Gaussian vector $(X, Y)\in \mathbb{R}^d\times\mathbb{R}^d$ and $k$-th ($|k|=1$) order derivative, we have
\begin{align*}
\mathbb{E}[f^{(k)}_{\varepsilon}(X)f^{(k)}_{\varepsilon}(Y)]&=\frac{d^2}{\varepsilon^2}(2\pi\varepsilon)^{-d}\mathbb{E}[X_1Y_1e^{-\frac{X_1^2+\cdots+X_d^2+Y_1^2+\cdots+Y_d^2}{2\varepsilon}}]\\
&=\frac{d^2}{\varepsilon^2}(2\pi\varepsilon)^{-1}\int_{\mathbb{R}^2}x_1y_2e^{-\frac{x_1^2+y_1^2}{2\varepsilon}}f_{\Sigma}(x_1,y_1)dx_1dy_1\\
&\qquad\times(2\pi\varepsilon)^{-(d-1)}\int_{\mathbb{R}^{2(d-1)}}e^{-\frac{x_2^2+y_2^2+\cdots+x_d^2+y_d^2}{2\varepsilon}}f_{\Sigma^{d-1}}(\widetilde{x},\widetilde{y})d\widetilde{x}d\widetilde{y}\\
&=d^2\varepsilon^{-2}(2\pi)^{-1}\varepsilon^2|\varepsilon I+\Sigma|^{-\frac{3}{2}}\Sigma_{1,2}\times(2\pi)^{-(d-1)}|\varepsilon I+\Sigma|^{-\frac{d-1}{2}}\\
&=d^2(2\pi)^{-d}|\varepsilon I+\Sigma|^{-\frac{d}{2}-1}\Sigma_{1,2},
\end{align*}
where $\Sigma=(\Sigma_{i,j})_{i,j=1,2}$ is the covariance matrix of $(B^{H,1}_s-B^{H,1}_r,B^{H,1}_{s'}-B^{H,1}_{r'})$, $\widetilde{x}=(x_2,\cdots,x_d), \widetilde{y}=(y_2,\cdots,y_d)$ and $\Sigma^{d-1}$ is the covariance matrix of $(\widetilde{B}^{H}_s-\widetilde{B}^{H}_r,\widetilde{B}^{H}_{s'}-\widetilde{B}^{H}_{r'})$ ($\widetilde{B}^H$ denotes the $(d-1)$-dimensional fBm).

Thus,
$$
\mathbb{E}\Big[\Big|\widehat{\alpha}^{'}_{t,\varepsilon}(0)\Big|^2\Big]=V_1(\varepsilon)+V_2(\varepsilon)+V_3(\varepsilon)
$$
and
\begin{align*}
V_i(\varepsilon)=\frac{2d^2}{(2\pi)^d}\int_{D_i}|\varepsilon I+\Sigma|^{-\frac{d}2-1}|\mu| drdsdr'ds'
\end{align*}
where $D_i$ (i=1, 2, 3) defined in Lemma \ref{sec3-lem3.2}
and $\Sigma$ is a covariance matrix with $\Sigma_{1,1}=\lambda$, $\Sigma_{2,2}=\rho$,  $\Sigma_{1,2}=\mu$ given in Lemma \ref{sec3-lem3.2}.

Next, we will split the proof into three parts to consider $V_1(\varepsilon)$, $V_2(\varepsilon)$ and $V_3(\varepsilon)$,  respectively.

\textbf{For the $V_1(\varepsilon)$ term}, changing the coordinates $(r, r', s, s')$ by $(r, a=r'-r, b=s-r', c=s'-s)$ and integrating the $r$ variable, we get
\begin{align*}
V_1(\varepsilon)
&\leq\,C\,\int_{[0,t]^4}|\varepsilon I+\Sigma|^{-\frac{d}2-1}|\mu| drdadbdc\\
&=\,C\,\int_{[0,t]^3}|\varepsilon I+\Sigma|^{-\frac{d}2-1}|\mu| dadbdc\\
&=:\widetilde{V_1}(\varepsilon).
\end{align*}

Since
$$|\mu|=\frac12\left|(a+b+c)^{2H}+b^{2H}-a^{2H}-c^{2H}\right|\leq\sqrt{\lambda\rho}=(a+b)^{H}(b+c)^{H}$$
and
\begin{align*}
|\varepsilon I+\Sigma|&=(\varepsilon+\Sigma_{1,1})(\varepsilon+\Sigma_{2,2})-\Sigma^2_{1,2}\\
&\geq C\Big[\varepsilon^2+\varepsilon((a+b)^{2H}+(b+c)^{2H})+a^{2H}(c+b)^{2H}+c^{2H}(a+b)^{2H}\Big]\\
&\geq C\Big[\varepsilon^2+(a+b)^{H}(b+c)^{H}(\varepsilon+(ac)^{H})\Big]\\
&\geq C(a+b)^{H}(b+c)^{H}(\varepsilon+(ac)^{H}),
\end{align*}
where we use the Young's inequality in the second to last inequality.

Then, we have
$$\widetilde{V_1}(\varepsilon)\leq C\int_{[0,t]^3}(a+b)^{-\frac{Hd}{2}}(b+c)^{-\frac{Hd}{2}}\Big(\varepsilon+(ac)^{H}\Big)^{-\frac{d}{2}-1}dadbdc.$$

We will estimate this integral over the regions $\{b\leq (a\vee c)\}$ and $\{b > (a\vee c)\}$ separately, and we will denote these two integrals by $\widetilde{V_{1,1}}$ and $\widetilde{V_{1,2}}$, respectively. If $b\leq (a\vee c)$, without loss of generality, we can assume $c\geq a$ and thus $b\leq c$. For a given small enough constant $\varepsilon_1>0$,

\begin{align*}
\widetilde{V_{1,1}}(\varepsilon)&\leq C\int_{[0,t]^3}(a+b)^{H-\frac{Hd}{2}}(a+b)^{-H}(b+c)^{-\frac{Hd}{2}}\Big(\varepsilon+(ac)^{H}\Big)^{-\frac{d}{2}-1}dadbdc\\
&\leq C\int_{[0,t]^3}b^{-H-\frac{Hd}2}a^{H-\frac{Hd}2}\Big(\varepsilon+(ac)^{H}\Big)^{-\frac{d}{2}-1}dadbdc\\
&\leq C \varepsilon^{\frac1H-\frac{d}2-1}\int_0^{t\varepsilon^{-\frac1{H}}}\int_0^ta^{H-\frac{Hd}2}\Big(1+(ac)^{H}\Big)^{-\frac{d}2-1}dadc,
\end{align*}
where we make the change of variable $c=c\,\varepsilon^{-\frac1H}$ in the last inequality.

By L'H\^{o}spital's rule, we have
\begin{align*}
\lim_{\varepsilon\to0}\widetilde{V_{1,1}}(\varepsilon)
&\leq\lim_{\varepsilon\to0}\frac{-\frac{Ct}{H}\varepsilon^{-1-\frac1H}\int_0^ta^{H-\frac12}(1+t^Ha^H\varepsilon^{-1})^{-\frac{d}2-1}da}{(1-\frac1{2H})\varepsilon^{-\frac1{2H}}}\\
&=\lim_{\varepsilon\to0}\frac{\frac{Ct}{H}}{\frac1{2H}-1}\int_0^{t\varepsilon^{-\frac1H}}a^{H-\frac12}(1+t^Ha^H)^{-\frac{d}2-1}da\\
&=O\left(\log\frac1{\varepsilon}\right),
\end{align*}
where we use Lemma \ref{sec3-lem-intlog} in the last equality.

If $b>(a\vee c)$, we can see that
$$\mu=\frac12((a+b+c)^{2H}+b^{2H}-a^{2H}-c^{2H})\leq C\, b^{2H}$$
and
\begin{align*}
|\varepsilon I+\Sigma|
&\geq C\Big[\varepsilon^2+\varepsilon((a+b)^{2H}+(b+c)^{2H})+a^{2H}(c+b)^{2H}+c^{2H}(a+b)^{2H}\Big]\\
&\geq C\, b^{2H}(\varepsilon+(a\vee c)^{2H}).
\end{align*}

Then
\begin{align*}
\limsup_{\varepsilon\to0}\frac{\widetilde{V_{1,2}}(\varepsilon)}{\log \frac1\varepsilon}&\leq \limsup_{\varepsilon\to0}\frac{C}{\log \frac1\varepsilon}\int_{[0,t]^3}b^{-Hd}\Big(\varepsilon+(a\vee c)^{2H}\Big)^{-\frac{d}{2}-1}dadbdc\\
&\leq \limsup_{\varepsilon\to0}\frac{C}{\log \frac1\varepsilon}\int_{[0,t]^3}b^{-2H}a^{2H-Hd}\Big(\varepsilon+(a\vee c)^{2H}\Big)^{-\frac{d}{2}-1}dadbdc\\
&\leq \limsup_{\varepsilon\to0}\frac{C}{\log \frac1\varepsilon}\int_0^t\int_0^aa^{2H-1}(\varepsilon+a^{2H})^{-\frac{d}2-1}dcda\\
&=\limsup_{\varepsilon\to0}\frac{C}{\log \frac1\varepsilon}\int_0^ta^{2H}(\varepsilon+a^{2H})^{-\frac{d}2-1}da<\infty,
\end{align*}
where we use Lemma \ref{sec3-lem-intlog} in the last inequality.

So, by the above result, we can obtain
\begin{equation}\label{sec3-eq3.7}
\lim_{\varepsilon\to0}\Big(\varepsilon^{-\frac1H}\log1/\varepsilon\Big)^{2H-1}V_{1}(\varepsilon)=0.
\end{equation}

\textbf{For the $V_2(\varepsilon)$ term}, changing the coordinates $(r, r', s, s')$ by $(r, a=r'-r, b=s'-r', c=s-s')$ and integrating the $r$ variable, we get
\begin{align*}
V_2(\varepsilon)
\leq\,c_{q,d}\int_{[0,t]^3}|\varepsilon I+\Sigma|^{-\frac32}|\mu| dadbdc=:\widetilde{V_2}(\varepsilon).
\end{align*}

By
\begin{align*}
|\mu|&=\frac12\Big((a+b)^{2H}+(b+c)^{2H}-a^{2H}-c^{2H}\Big)\\
&=Hb\int_0^1\left((a+bv)^{2H-1}+(c+bv)^{2H-1}\right)dv\\
&\leq \left\{
\begin{aligned}
   2Hb^{2H},~~\qquad &\text{if} ~b\leq(a\vee c),  \\
   2Hb(a\wedge c)^{2H-1},~~\qquad &\text{if} ~b>(a\vee c).  \\
   \end{aligned}
\right.
\end{align*}
and
$$|\varepsilon I+\Sigma|=(\varepsilon+\Sigma_{1,1})(\varepsilon+\Sigma_{2,2})-\Sigma^2_{1,2}\geq \varepsilon^2+\varepsilon((a+b+c)^{2H}+b^{2H})+C\, b^{2H}(a^{2H}+c^{2H}),$$
we have
$$\widetilde{V_2}(\varepsilon)\leq C\int_{[0,t]^3}\mu\Big(\varepsilon((a+b+c)^{2H}+b^{2H})+ b^{2H}(a^{2H}+c^{2H})\Big)^{-\frac{d}2-1}dadbdc.$$

We again estimate this integral over the regions $\{b\leq (a\vee c)\}$ and $\{b > (a\vee c)\}$ separately, and denote these two integrals by $\widetilde{V_{2,1}}$ and $\widetilde{V_{2,2}}$, respectively. If $b\leq (a\vee c)$,
\begin{align*}
\widetilde{V_{2,1}}(\varepsilon)&\leq C\int_{[0,t]^3}b^{2H}\Big(\varepsilon(a\vee c)^{2H}+b^{2H}(a\vee c)^{2H}\Big)^{-\frac{d}{2}-1}dadbdc\nonumber\\
&\leq C\int_{[0,t]^3}(a\vee c)^{-1-2H}b^{2H}\Big(\varepsilon+b^{2H}\Big)^{-\frac{d}2-1}dadbdc\nonumber\\
&\leq C \int_0^tb^{2H}\Big(\varepsilon+b^{2H}\Big)^{-\frac{d}2-1}db\nonumber\\
&=O(\log \frac1\varepsilon), ~as~ \varepsilon\to0,
\end{align*}
where we use Lemma \ref{sec3-lem-intlog} in the last equality.

If $b>(a\vee c)$, similarly, we have
\begin{align*}
\limsup_{\varepsilon\to0}\frac{\widetilde{V_{2,2}}(\varepsilon)}{\log \frac1\varepsilon}
&\leq \limsup_{\varepsilon\to0}\frac{C}{\log \frac1\varepsilon}\int_{[0,t]^3}b(a\wedge c)^{2H-1}[b^{2H}(\varepsilon+(a\vee c)^{2H})]^{-\frac{d}2-1}dadbdc\\
&\leq \limsup_{\varepsilon\to0}\frac{C}{\log \frac1\varepsilon}\int_0^tb^{-2H}db\int_{[0,t]^2}(a\wedge c)^{2H-1}(\varepsilon+(a\vee c)^{2H})]^{-\frac{d}2-1}dcda\\
&\leq \limsup_{\varepsilon\to0}\frac{C}{\log \frac1\varepsilon}\int_0^t\int_0^a c^{2H-1}[\varepsilon+a^{2H}]^{-\frac{d}2-1}dcda\\
&=\limsup_{\varepsilon\to0}\frac{C}{\log \frac1\varepsilon}\int_0^ta^{2H}(\varepsilon+a^{2H})^{-\frac{d}2-1}da<\infty.
\end{align*}
So, by the above result, we can obtain
\begin{align}\label{sec3-eq3.8}
\lim_{\varepsilon\to0}\Big(\varepsilon^{-\frac1H}\log1/\varepsilon\Big)^{2H-1}V_{2}(\varepsilon)=0.
\end{align}

\textbf{For the $V_3(\varepsilon)$ term}.
$$V_3(\varepsilon)
=\frac{2d^2}{(2\pi)^d}\int_{D_3}|\varepsilon I+\Sigma|^{-d/2-1}|\mu| dsdrds'dr'$$
We first change the coordinates $(r, r', s, s')$ by $(r, a=s-r, b=r'-s, c=s'-r')$ and then by
\begin{align*}
|\mu|&=\frac12\Big|(a+b+c)^{2H}+b^{2H}-(b+c)^{2H}-(a+b)^{2H}\Big|\\
&=H(1-2H)ac\int_0^1\int_0^1(b+ax+cy)^{2H-2}dxdy\\
&=:\mu(a+b,a,c),
\end{align*}
and $|\varepsilon I+\Sigma|=\varepsilon^2+\varepsilon(a^{2H}+c^{2H})+(ac)^{2H}-\mu(a+b,a,c)^2$,
we can find
\begin{align*}
V_3(\varepsilon)
&=\frac{2d^2}{(2\pi)^d}\int_{[0,t]^3}\mathds{1}_{(0,t)}(a+b+c)(t-a-b-c)|\varepsilon I+\Sigma|^{-d/2-1}|\mu| dadbdc\\
&=\frac{2d^2}{(2\pi)^d} \int_{[0,t\varepsilon^{-\frac1{2H}}]^2\times[0,t]}\mathds{1}_{(0,t)}(b+\varepsilon^{\frac1{2H}}(a+c))\\
&\qquad\qquad\times \frac{(t-b-\varepsilon^{\frac1{2H}}(a+c))\mu(\varepsilon^{\frac1{2H}}a+b,\varepsilon^{\frac1{2H}}a,\varepsilon^{\frac1{2H}}c)}{\Big[(1+a^{2H})(1+c^{2H})-\varepsilon^{-2}\mu(\varepsilon^{\frac1{2H}}a+b,\varepsilon^{\frac1{2H}}a,\varepsilon^{\frac1{2H}}c)^2\Big]^{\frac{d}{2}+1}}
\varepsilon^{\frac1{H}-2(d/2+1)}dbdadc,
\end{align*}
where we change the coordinates $(a,b,c)$ by $(\varepsilon^{-\frac1{2H}}a, b, \varepsilon^{-\frac1{2H}}c)$ in the last equality.

By the definition of $\mu(a+b,a,c)$, it is easy to find
\begin{align*}
\mu(\varepsilon^{\frac1{2H}}a+b,\varepsilon^{\frac1{2H}}a,\varepsilon^{\frac1{2H}}c)&= H(1-2H)
\varepsilon^{\frac1{H}}ac\int_{[0,1]^2}(b+\varepsilon^{\frac1{2H}}av_1+\varepsilon^{\frac1{2H}}cv_2)^{2H-2}dv_1dv_2
\end{align*}
and
$$\varepsilon^{-\frac1{H}}\mu(\varepsilon^{\frac1{2H}}a+b,\varepsilon^{\frac1{2H}}a,\varepsilon^{\frac1{2H}}c)=H(1-2H)acb^{2H-2}+O(\varepsilon^{\frac1{2H}}ac(a+c)).$$

The other part of the integrand in $V_3(\varepsilon)$ is
\begin{align*}
&\Big[(1+a^{2H})(1+c^{2H})-\varepsilon^{-2}\mu(\varepsilon^{\frac1{2H}}a+b,\varepsilon^{\frac1{2H}}a,\varepsilon^{\frac1{2H}}c)^2\Big]^{-\frac{d}2-1}\\
&\qquad=\Big[(1+a^{2H})(1+c^{2H})\Big]^{-\frac{d}2-1}+O\Big(\varepsilon^{\frac2H-2} a^2c^2[(1+a^{2H})(1+c^{2H})]^{-\frac{d}2-3}\Big).
\end{align*}

Let $O_{\varepsilon,3}=\{[0,t\varepsilon^{-\frac1{2H}}]^2\times[(\log\frac1{\varepsilon})^{-1},t]\}$ and
\begin{align*}
\widetilde{V_3}(\varepsilon)
&=\frac{2d^2}{(2\pi)^d}\int_{O_{\varepsilon,3}}\mathds{1}_{(0,t)}(b+\varepsilon^{\frac1{2H}}(a+c))\\
&\qquad\qquad\times \frac{(t-b-\varepsilon^{\frac1{2H}}(a+c))\mu(\varepsilon^{\frac1{2H}}a+b,\varepsilon^{\frac1{2H}}a,\varepsilon^{\frac1{2H}}c)}{\Big[(1+a^{2H})(1+c^{2H})-\varepsilon^{-2}\mu(\varepsilon^{\frac1{2H}}a+b,\varepsilon^{\frac1{2H}}a,\varepsilon^{\frac1{2H}}c)^2\Big]^{\frac{d}{2}+1}}
\varepsilon^{\frac1{H}-2(d/2+1)}dbdadc.
\end{align*}

Note that
\begin{align*}
&\limsup_{\varepsilon\to0}\Big(\varepsilon^{-\frac1H}\log1/\varepsilon\Big)^{2H-1}|V_3(\varepsilon)-\widetilde{V_3}(\varepsilon)|\\
&\leq \limsup_{\varepsilon\to0}c_{H,d}\Big(\varepsilon^{-\frac1H}\log1/\varepsilon\Big)^{2H-1}\int_{[0,t\varepsilon^{-\frac1{2H}}]^2\times[0,(\log\frac1{\varepsilon})^{-1}]}\varepsilon\mu(a,a,c)\\
&\qquad\qquad\qquad\times\Big[(1+a^{2H})(1+c^{2H})-\mu^2(a,a,c)\Big]^{-\frac{d}{2}-1}
\varepsilon^{\frac1{H}-2(d/2+1)}dbdadc\\
&\leq \limsup_{\varepsilon\to0}c_{H,d}\Big(\varepsilon^{-\frac1H}\log1/\varepsilon\Big)^{2H-1}(\varepsilon\log\frac1{\varepsilon})^{-1}\int_{[0,t\varepsilon^{-\frac1{2H}}]^2}(a+c)^{2H}\Big[1+(ac)^{2H}-(a+c)^{2H}\Big]^{-\frac{d}{2}-1}dadc\\
&=0.
\end{align*}

Thus,
\begin{equation}\label{sec3-eq V-V3}
\lim_{\varepsilon\to0}\Big(\varepsilon^{-\frac1H}\log1/\varepsilon\Big)^{2H-1}V_3(\varepsilon)=\lim_{\varepsilon\to0}\Big(\varepsilon^{-\frac1H}\log1/\varepsilon\Big)^{2H-1}\widetilde{V_3}(\varepsilon).
\end{equation}

Since
$$\Big(\log1/\varepsilon\Big)^{2H-1}\int_{(\log\frac1{\varepsilon})^{-1}}^tb^{2H-2}db<\infty,$$
\begin{equation}\label{sec3-eq3.11-}
\begin{split}
&\Big(\varepsilon^{-\frac1H}\Big)^{2H-1}\int_{[0,t\varepsilon^{-\frac1{2H}}]^2}\varepsilon^{\frac1{2H}+\frac2H-d-2}ac(a+c)\Big[(1+a^{2H})(1+c^{2H})\Big]^{-\frac{d}2-1}dadc\\
&\qquad+\Big(\varepsilon^{-\frac1H}\Big)^{2H-1}\int_{[0,t\varepsilon^{-\frac1{2H}}]^2}\varepsilon^{\frac2H-2+\frac2H-d-2} a^3c^3\Big[(1+a^{2H})(1+c^{2H})\Big]^{-\frac{d}2-3}dadc\\
&\to0,
\end{split}
\end{equation}
as $\varepsilon\to0$. Then, by L'H\^{o}spital's rule, we have
\begin{equation}\label{sec3-eq3.11}
\begin{split}
&\lim_{\varepsilon\to0}\Big(\varepsilon^{-\frac1H}\log1/\varepsilon\Big)^{2H-1}\widetilde{V_3}(\varepsilon)\\
&=H(1-2H)\frac{2d^2}{(2\pi)^d}\lim_{\varepsilon\to0}\Big(\log1/\varepsilon\Big)^{2H-1}\int_{(\log\frac1{\varepsilon})^{-1}}^t(t-b)b^{2H-2}db\\
&\qquad\times\lim_{\varepsilon\to0}\Big(\varepsilon^{-\frac1H}\Big)^{2H-1}\int_{[0,t\varepsilon^{-\frac1{2H}}]^2}ac\Big[(1+a^{2H})(1+c^{2H})\Big]^{-\frac{d}2-1}dadc\\
&=H(1-2H)\frac{2d^2}{(2\pi)^d}\times\frac{t}{1-2H}\times\frac{t^{2-4H}}{(1-2H)^2}.
\end{split}
\end{equation}

Together \eqref{sec3-eq3.7}, \eqref{sec3-eq3.8}, \eqref{sec3-eq V-V3} and \eqref{sec3-eq3.11}, we can see
$$
\lim_{\varepsilon\to0}\mathbb{E}\Big[\Big|\Big(\varepsilon^{-\frac1H}\log1/\varepsilon\Big)^{H-1/2}\widehat{\alpha}^{'}_{t,\varepsilon}(0)\Big|^2\Big]
=\frac{2Hd^2t^{3-4H}}{(2\pi)^d(1-2H)^2}=:\sigma^2.
$$
\end{proof}

\begin{lemma} \label{sec3-lem3.5}
For $I_1(f_{1,\varepsilon})$ given in \eqref{sec2-eq2.3}, then
$$
\lim_{\varepsilon\to0}\mathbb{E}\Big[\Big|\Big(\varepsilon^{-\frac1H}\log1/\varepsilon\Big)^{H-1/2}I_1(f_{1,\varepsilon})\Big|^2\Big]=\sigma^2.
$$
\end{lemma}

\begin{proof}
Form \eqref{sec2-eq2.3}, we can find
\begin{equation} \label{sec3-eq3.lem32}
\mathbb{E}\Big[\Big|I_1(f_{1,\varepsilon})\Big|^2\Big]=\Big(V_1^{(1)}(\varepsilon)+V_2^{(1)}(\varepsilon)+V_3^{(1)}(\varepsilon)\Big),
\end{equation}
where $V_i^{(1)}(\varepsilon)=2\int_{D_i}\langle f_{1,\varepsilon,s_1,r_1},f_{1,\varepsilon,s_2,r_2}\rangle_{\mathfrak{H}}dr_1dr_2ds_1ds_2$
for $i=1, 2, 3$, and $\langle f_{1,\varepsilon,s_1,r_1},f_{1,\varepsilon,s_2,r_2}\rangle_{\mathfrak{H}}$ was defined in \eqref{sec2-eq2.ine}. Then we have
\begin{equation} \label{sec3-eq3.15}
0\leq V_i^{(1)}(\varepsilon)\leq V_i(\varepsilon).
\end{equation}

Combining  \eqref{sec3-eq3.15} with \eqref{sec3-eq3.7} and \eqref{sec3-eq3.8}, we can see
$$
\lim_{\varepsilon\to0}\Big(\varepsilon^{-\frac1H}\log1/\varepsilon\Big)^{2H-1}\Big(V_1^{(1)}(\varepsilon)+V_2^{(1)}(\varepsilon)\Big)=0.
$$
Thus, we only need to consider $\Big(\varepsilon^{-\frac1H}\log1/\varepsilon\Big)^{2H-1}V_3^{(1)}(\varepsilon)$ as $\varepsilon\to0$.

By \eqref{sec2-eq2.3} and \eqref{sec3-eq3.lem32} we have
\begin{align*}
V_3^{(1)}(\varepsilon)&=2\beta_{1,d}^2\int_{S_3}G^{(1)}_{\varepsilon,r'-r}(s-r,s'-r')\\
&=2\beta_{1,d}^2\int_{[0,t]^3}\int_0^{t-(a+b+c)}\mathds{1}_{(0,t)}(a+b+c)(\varepsilon+a^{2H})^{-d/2-1}(\varepsilon+c^{2H})^{-d/2-1}\mu(a+b,a,c)ds_1dadbdc\\
&=2H(1-2H)\beta_{1,d}^2\int_0^t\int_{[0,t\varepsilon^{-\frac1{2H}}]^2}\int_{[0,1]^2}\mathds{1}_{(0,t)}\Big((b+\varepsilon^{\frac1{2H}}(a+c)\Big)\Big(t-b-\varepsilon^{\frac1{2H}}(a+c)\Big)\\
&\qquad\qquad\qquad \times \Big[(1+a^{2H})(1+c^{2H})\Big]^{-d/2-1}ac\Big(b+\varepsilon^{\frac1{2H}}(av_1+cv_2)\Big)^{2H-2}dv_1dv_2dadcdb.
\end{align*}
Note that
\begin{align*}
\int_{[0,1]^2}\Big(b+\varepsilon^{\frac1{2H}}(av_1+cv_2)\Big)^{2H-2}dv_1dv_2=b^{2H-2}+O(\varepsilon^{\frac1{2H}}(a+c))
\end{align*}
and
\begin{align*}
\int_{[0,1]^2}&\Big(t-b-\varepsilon^{\frac1{2H}}(a+c)\Big) \Big[(1+a^{2H})(1+c^{2H})\Big]^{-d/2-1}ac\Big(b+\varepsilon^{\frac1{2H}}(av_1+cv_2)\Big)^{2H-2}dv_1dv_2\\
&=(t-b)b^{2H-2}ac\Big[(1+a^{2H})(1+c^{2H})\Big]^{-d/2-1}+O\left(\varepsilon^{\frac1{2H}}(a+c)ac\Big[(1+a^{2H})(1+c^{2H})\Big]^{-d/2-1}\right).
\end{align*}

Similar to \eqref{sec3-eq V-V3},
$$
\lim_{\varepsilon\to0}\Big(\varepsilon^{-\frac1H}\log1/\varepsilon\Big)^{2H-1}V^{(1)}_3(\varepsilon)=\lim_{\varepsilon\to0}\Big(\varepsilon^{-\frac1H}\log1/\varepsilon\Big)^{2H-1}\widetilde{V}^{(1)}_3(\varepsilon),
$$
where
\begin{align*}
|\widetilde{V}_3^{(1)}(\varepsilon)|
&=2H(1-2H)\beta_{1,d}^2\int_{(\log\frac1\varepsilon)^{-1}}^t\int_{[0,t\varepsilon^{-\frac1{2H}}]^2}\int_{[0,1]^2}\mathds{1}_{(0,t)}\Big((b+\varepsilon^{\frac1{2H}}(a+c)\Big)\Big(t-b-\varepsilon^{\frac1{2H}}(a+c)\Big)\\
&\qquad\qquad\qquad \times \Big[(1+a^{2H})(1+c^{2H})\Big]^{-d/2-1}ac\Big(b+\varepsilon^{\frac1{2H}}(av_1+cv_2)\Big)^{2H-2}dv_1dv_2dadcdb.
\end{align*}

According to \eqref{sec3-eq3.11-} and \eqref{sec3-eq3.11}, we can find that
\begin{align*}
&\lim_{\varepsilon\to0}\Big(\varepsilon^{-\frac1H}\log1/\varepsilon\Big)^{2H-1}\widetilde{V}^{(1)}_3(\varepsilon)\\
&=2H(1-2H)\beta^2_{1,d}\lim_{\varepsilon\to0}\Big(\log1/\varepsilon\Big)^{2H-1}\int_{(\log\frac1{\varepsilon})^{-1}}^t(t-b)b^{2H-2}db\\
&\qquad\times\lim_{\varepsilon\to0}\Big(\varepsilon^{-\frac1H}\Big)^{2H-1}\int_{[0,t\varepsilon^{-\frac1{2H}}]^2}ac\Big[(1+a^{2H})(1+c^{2H})\Big]^{-\frac{d}2-1}dadc\\
&=H(1-2H)\frac{2d^2}{(2\pi)^d}\times\frac{t}{1-2H}\times\frac{t^{2-4H}}{(1-2H)^2}=\sigma^2,
\end{align*}
where we use $\beta_{1,d}^2=\frac{d^2}{(2\pi)^d}$ in the second equality.

Thus,
\begin{align*}
\lim_{\varepsilon\to0}\Big(\varepsilon^{-\frac1H}\log1/\varepsilon\Big)^{2H-1}V^{(1)}_3(\varepsilon)=\sigma^2.
\end{align*}
\end{proof}

\textbf{Proof of  Theorem  \ref{sec1-th.d=3-2}}

By Lemmas \ref{sec3-lem3.4}--\ref{sec3-lem3.5} and
$$\widehat{\alpha}^{'}_{t,\varepsilon}(0)=I_1(f_{1,\varepsilon})+\sum_{q=2}^\infty I_{2q-1}(f_{2q-1,\varepsilon}),$$
we can see
$$\lim_{\varepsilon\to0}\mathbb{E}\Big[\Big|\Big(\varepsilon^{-\frac1H}\log1/\varepsilon\Big)^{H-1/2}\sum_{q=2}^\infty I_{2q-1}(f_{2q-1,\varepsilon})\Big|^2\Big]=0.$$

Since $I_1(f_{1,\varepsilon})$ is Gaussian, then we have,  as $\varepsilon\to0$,
$$
\Big(\varepsilon^{-\frac1H}\log1/\varepsilon\Big)^{H-1/2}I_1(f_{1,\varepsilon})\overset{law}{\to}N(0,\sigma^2).
$$
Thus,
$$
\Big(\varepsilon^{-\frac1H}\log1/\varepsilon\Big)^{H-1/2}\widehat{\alpha}^{'}_{t,\varepsilon}(0)\overset{law}{\to}N(0,\sigma^2),
$$
as $\varepsilon\to0$. This completes the proof.

\bigskip



\textbf{Declaration of interests} ~The authors declare that they have no known competing financial interests or personal relationships that
could have appeared to influence the work reported in this paper.

\end{document}